\documentclass[11pt, a4paper]{amsart}
\usepackage[T1]{fontenc}
\usepackage[utf8]{inputenc}
\usepackage{amsmath,amssymb,amsfonts}
\usepackage[english]{babel}

\usepackage[all,cmtip]{xy}
\usepackage{booktabs}
\usepackage{stmaryrd}

\usepackage{tikz, tikz-cd}
\usepackage{hyperref}
\usepackage{todonotes}

\numberwithin{equation}{section}

\newcommand{\laxtimes}[1]{\mathop{\times\mkern-13mu\raise1.3ex\hbox{$\scriptscriptstyle\to$}_{#1}}}

\DeclareMathOperator{\Hom}{Hom}

\DeclareMathOperator{\id}{id}

\DeclareMathOperator*{\colim}{colim}

\DeclareMathOperator{\Cris}{Cris}
\DeclareMathOperator{\Sym}{Sym}

\newcommand{\cC}{\mathcal{C}}
\newcommand{\cD}{\mathcal{D}}
\newcommand{\D}{\mathrm{D}}

\newcommand{\FSmAlg}{\mathrm{FSmAlg}}
\newcommand{\FSmAff}{\mathrm{FSmAff}}

\newcommand{\PFSmAlg}{\mathrm{PFSmAlg}}
\newcommand{\SmAlg}{\mathrm{SmAlg}}
\newcommand{\SmAff}{\mathrm{SmAff}}
\newcommand{\Sm}{\mathrm{Sm}}
\newcommand{\fSm}{\mathrm{FSm}}
\newcommand{\Ch}{\mathrm{Ch}}

\newcommand{\bc}{\mathrm{BC}}

\DeclareMathOperator{\Fun}{Fun}

\DeclareMathOperator{\Nerve}{N}
\DeclareMathOperator{\Map}{Map}

\newcommand{\sSet}{\mathbf{sSet}}

\newcommand{\Sh}{\mathrm{Sh}}
\newcommand{\PSh}{\mathrm{PSh}}

\newcommand{\op}{\mathrm{op}}

\newcommand{\dR}{\mathrm{dR}}

\newcounter{zaehler}
\theoremstyle{plain}

\theoremstyle{plain}
\newtheorem{thm}{Theorem}[section]

\newtheorem{lem}[thm]{Lemma}
\newtheorem{lemma}[thm]{Lemma}

\newtheorem{prop}[thm]{Proposition}

\newtheorem*{conj*}{Conjecture}

\theoremstyle{definition}
\newtheorem{ex}[thm]{Example}
\newtheorem{defn}[thm]{Definition}

\newtheorem{claim}[thm]{Claim}
\newtheorem{rmk}[thm]{Remark}

\title{A remark on crystalline cohomology}

\author{Moritz Kerz}
\address{Fakult\"at f\"ur Mathematik, Universit\"at Regensburg, 93040 Regensburg, Germany}
\email{moritz.kerz@ur.de}

\author{Georg Tamme}
\address{Institut f\"ur Mathematik, Fachbereich 08, Johannes Gutenberg-Universit\"at Mainz, D-55099 Mainz, Germany}
\email{georg.tamme@uni-mainz.de}
\thanks{MK is partially supported by the DFG through CRC 1085 \textit{Higher Invariants} (Universit\"at Regensburg). GT is partially supported by the DFG through TRR 326 (Project-ID 444845124).}



\date{\today}

\begin{document}

\begin{abstract}
We propose a new approach to crystalline cohomology based on the observation that one can lift smooth algebras uniquely ``up to coherent homotopy.''
\end{abstract}

\maketitle

\section{Introduction}

In this note, we explore an approach to crystalline cohomology using certain homotopies
congruent to the identity. In the introduction we present the idea in its simplest
setting. Let $k$ be a
perfect field of characteristic $p > 0$ and let $R$ be its ring of Witt vectors. For a
smooth $k$-algebra $\overline{A}$, a classical approach to defining a good cohomology
theory, see, for example, \cite{MonskyWashnitzer,GrothendieckCrystals}, is to lift $\overline A$  to a $p$-adic
complete smooth $R$-algebra $A$ and to consider the $p$-complete de Rham cohomology of $A$
over $R$.

The issue with this naive approach is the difficulty of gluing this cohomology theory for more global objects over $k$. Indeed, the lift $A$ over $R$ for a given $\overline{A}$ is unique only up to non-unique isomorphism. This non-uniqueness causes significant problems. To address this, consider two lifts $A_1$ and $A_2$ of $\overline{A}$ and two isomorphisms
\[
\phi_1, \phi_2\colon A_1 \xrightarrow{\sim} A_2
\]
which are congruent to the identity modulo $p$. Using the smoothness of $A_1$, one can construct a formal homotopy
\[
h\colon A_1 \to A_2\llbracket T \rrbracket
\]
such that the composition of $h$ with the  map $T \mapsto 0$ is $\phi_1$, and the composition of $h$ with the  map $T \mapsto p$ is $\phi_2$. Using a Poincaré lemma, one then shows that $\phi_1$ and $\phi_2$ induce the same map on de Rham cohomology in the derived category of $p$-complete $R$-modules $\D(R)^\wedge$.

However, the next challenge in globalizing this construction is that $h$ is again not unique. Nevertheless, it is unique up to a certain formal homotopy, and so on. The aim of this note is to make this idea precise using the language of higher category theory, thereby providing a new approach to crystalline cohomology.

In this note, we do not discuss the notion of crystals within this framework but focus
solely on crystalline cohomology with trivial coefficients.

Let us remark that there is a methodologically related approach to derived crystalline cohomology in Mao's paper \cite{MaoDerived} using free simplicial resolutions. 

\medskip

\noindent \textbf{Acknowledgement.} 
This project arose from discussions with Shuji Saito and this note profited a lot from
his input. Akhil Mathew suggested the line of argument in
  Section~\ref{sec:simplicial} to us thereby simplifying our original argument considerably. 
We thank Niklas Kipp and Sebastian Wolf for  discussions which led to Remark~\ref{rmk:not-loc}, and Federico Binda, Denis-Charles Cisinski and Kay Rülling for valuable comments and discussions. Finally, we would like to thank the referee for helpful comments.

\medskip

\noindent \textbf{Notation.} We denote by $\Delta[m] \in \sSet$ the standard, simplicial $m$-simplex, i.e.~the functor represented by $[m]$. We denote by $\partial\Delta[m]$ its boundary.

\section{A simplicial ring}
	\label{sec:simplicial}

Let $R$ be a ring and let $\pi\in R$ be a non-zero divisor. Set $\overline R=R/(\pi)$.  We
assume that $\overline R\ne 0$.
Consider the simplicial ring
$R\Delta_\bullet$ with $R\Delta_m = R\llbracket T_0, \ldots , T_m\rrbracket$ and with
\[
R\Delta_m(\sigma)     \colon R\Delta_m \to R \Delta_n,\quad R\Delta_m(\sigma) ( T_i) = \sum_{j\in \sigma^{-1}(i)} T_j
\]
for a morphism $\sigma\colon [n]\to [m]$.

\begin{lem}\label{lem:simplring1}
The canonical map $ R\Delta_\bullet \xrightarrow{\sim} R$ induced by $T_i \mapsto 0$ is a trivial Kan fibration.
\end{lem}

\begin{proof}
Recall that any surjection of simplicial groups is a Kan fibration. 
Consider the short exact sequence of simplicial groups
  \[
   0 \to  I_\bullet \to R \Delta_\bullet \to  R \to 0.
 \]
 We have to show that $\pi_i(I_\bullet)=0$ for all $i\ge 0$. Let $S_\bullet$ be the free
 simplicial $R$-module with $S_m=R^{[m]}$. It is well known that $S_\bullet$ (often denoted $ER$) is
 contractible. One way to see this is to observe that $S_{\bullet}$ is the nerve of the category with set of objects given by $R$ and with a unique morphism between any two objects. Since this category has an initial object, its nerve is contractible. It follows that $I_\bullet = \prod_{n>0} \mathrm{Sym}_R^n(S_\bullet)$ is also
 contractible as a simplicial $R$-module.
\end{proof}

Consider the simplicial ring $R\Delta^\pi_\bullet$ with
\[
  R \Delta^\pi_m=R\Delta_m/(T_0 +
  \cdots + T_m -\pi).
\]

\begin{lem}\label{lem:simplring2}
The canonical map $R\Delta^\pi_\bullet \xrightarrow{\sim} \overline R$ induced by
$T_i\mapsto 0$ is a trivial Kan fibration.
\end{lem}

\begin{proof}
  Consider the following commutative diagram of simplicial $R$-modules with exact rows.
  \[
  \begin{tikzcd}[column sep = 1.7cm]
  0 \ar[r] &  R\Delta_\bullet \ar[r, "\sum_i T_i  -\pi"] \ar[d, "\wr"] & R\Delta_\bullet \ar[r] \ar[d, "\wr"] &  R\Delta^\pi_\bullet \ar[r] \ar[d, "\wr"] & 0 \\
0 \ar[r] &  R \ar[r, "-\pi"] & R \ar[r] & \overline R \ar[r] & 0 
  \end{tikzcd}
  \]
  The two vertical arrows on the left are equivalences by
Lemma~\ref{lem:simplring1}, so the right vertical map is an equivalence too.
\end{proof}

Set
\[
  R\partial\Delta_m^\pi = R\Delta_m^\pi  /(T_0 \cdots T_m)
\]
for $m\ge 0$.
Recall that for a simplicial set $S_\bullet$ and a simplicial commutative ring $A_\bullet$
the set of morphisms of simplicial sets $\Hom(S_\bullet, A_\bullet)$ forms a commutative ring itself.

\begin{lem}\label{lem:simplring3}
  For any $m\ge 1$ the canonical map
  \begin{equation}\label{eq:lemisosi}
    R\partial\Delta_m^\pi \xrightarrow{\sim} \Hom(\partial\Delta[m] , R\Delta_\bullet^\pi)
    \times_{ \Hom(\partial\Delta[m] , \overline R)}  \Hom(\Delta[m] , \overline R)
  \end{equation}
  is an isomorphism. The canonical map $ R\partial\Delta_0^\pi \xrightarrow{\sim}
  \overline R$ is an isomorphism.
\end{lem}

\begin{proof}
The second assertion is immediate.
For the first one, note that we have a canonical isomorphism $R\Delta^{\pi}_{m} =
\Hom(\Delta[m], R\Delta^{\pi}_{\bullet})$. As $R\Delta^{\pi}_{\bullet} \to \overline R$ is
a trivial Kan fibration by Lemma \ref{lem:simplring2}, the restriction map
\[
  R\Delta^{\pi}_{m} = \Hom(\Delta[m], R\Delta^{\pi}_{\bullet}) \to \Hom(\partial\Delta[m],
  R\Delta^{\pi}_{\bullet}) \times_{ \Hom(\partial\Delta[m] , \overline R)}  \Hom(\Delta[m] , \overline R)
\]
is surjective for $m\geq 1$. An element $f \in R\Delta^{\pi}_{m}$ lies in its kernel if and only if $\partial_{i}(f) = 0$ for $i=0, \dots, m$.

Observe that, as $\pi$ is a non-zero divisor in $R$, so is $\sum_{i} T_{i} -\pi$ in
$R\Delta_{m}$. Using this, one sees that any subsequence of $T_{0}, \dots, T_{m}, \sum_{i}
T_{i} -\pi$ is a regular sequence in $R\Delta_{m}$. By Claim~\ref{claim:refseqcl} below this
implies that any permutation of $T_{0}, \dots, T_{m}$ forms a regular sequence in $R\Delta^{\pi}_{m}$.

For $f \in R\Delta^{\pi}_{m}$ we get that  $\partial_{i}(f) =
0$ if and only if $f$ is divisible by $T_{i}$.  By Claim~\ref{claim:refseqcl}  this finally implies 
that $\partial_{i}(f)$ vanishes for all $i=0, \dots, m$ if and only if $f$ is divisible by the product $T_0 \cdots T_m$. 
This finishes the proof of the lemma.
\end{proof}

\begin{claim}\label{claim:refseqcl}
Let $A$ be a ring and $\mathbf a = (a_1,\ldots , a_r)\in A^r$ be sequence of elements which do not
generate the unit ideal. Then we have the implications $(1)\Leftrightarrow (2) \Rightarrow
(3)$ among the  properties:
\begin{itemize}
\item[(1)] any permutation of $\mathbf a$ is a regular sequence, 
\item[(2)] any subsequence of $\mathbf a$ is a regular sequence,
  \item[(3)] any element of $A$ divisible by $a_1,\ldots , a_r$ is divisible by $a_1
    \cdots a_r$.
\end{itemize}
\end{claim}
For the equivalence of (1) and (2) in the claim see \cite[Lemma 07DW]{stacks-project}. It is an easy exercise to check that (1) implies (3).

\section{Lifting smooth algebras}\label{sec:liftsmo}

Let the notation be as in Section~\ref{sec:simplicial} and assume additionally
that $R$ is $\pi$-adically complete. Let $\SmAlg_{\overline R}$ be the category
of  smooth $\overline R$-algebras and let $\FSmAlg_{ R}$ be the category of
$\pi$-adically complete topological $R$-algebras of finite type which are formally smooth over
$R$.

We consider the simplicial enrichment $\FSmAlg^\Delta_{ R}$ of
the category $\FSmAlg_{ R}$ by defining the mapping simplicial sets by
\[
\Hom^\Delta_R(A, B) := \Hom_R(A,B\Delta^\pi_\bullet) = \Hom_{R\Delta^\pi_\bullet}
(A\Delta^\pi_\bullet , B \Delta^\pi_\bullet)
\]
for $A,B\in \FSmAlg_{ R}$. Then we obtain simplicially enriched functors
\[
\FSmAlg_{ R}  \to \FSmAlg^\Delta_{ R} \xrightarrow{\bc} \SmAlg_{ \overline R}
\]
where the categories on the left and on the right are endowed with the trivial simplicial enrichment, the first functor is the inclusion of the underlying category of zero simplices, and $\bc$ is the base change functor given by $A \mapsto \overline A = A/\pi$ on objects and the obvious map on mapping simplicial sets.

\begin{prop}\label{prop:BCequiv}
The functor $\bc$ above is an equivalence of simplicial categories. Moreover, every mapping simplicial set in $\FSmAlg_{R}^{\Delta}$ is Kan. 
\end{prop}
We remark that the second assertion means that the simplicial category $\FSmAlg_{R}^{\Delta}$ is fibrant in the Bergner model structure on simplicial categories~\cite{Bergner}, though we will not make use of this model structure.

\begin{proof}
The essential surjectivity follows from the fact that the obstruction to the deformation
of a smooth algebra vanishes. In order to prove the proposition we will show that for
$A,B\in \FSmAlg_{ R}$ the map of simplicial sets
\[
 S_\bullet:= \Hom_R(A,B\Delta^\pi_\bullet) \to \Hom_{\overline R}(\overline A,\overline B) =: T_\bullet
\]
is a trivial Kan fibration.
This means we have to show that the canonical map
\[
S_m \to \Hom(\partial \Delta[m], S_\bullet ) \times_{\Hom(\partial \Delta[m], T_\bullet )} T_m
\]
is surjective for all $m\ge 0$. For $m\ge 1$ Lemma~\ref{lem:simplring3} tells us that this
map is isomorphic to the map
\[
 \Hom_R(A,B\Delta^\pi_m)  \to  \Hom_R(A,B\partial\Delta^\pi_m) .
\]
This map is surjective as $ B\Delta^\pi_m\to  B\partial \Delta^\pi_m$ is a
surjection whose domain is complete with respect to the kernel ideal generated by $T_0\cdots T_m$
and as $A$ is formally smooth over $R$.

For $m=0$, we have to check that $S_{0} = \Hom_{R}(A,B) \to T_{0} \cong \Hom_{R}(A, \overline B)$ is surjective. This follows from the formal smoothness of $A$ over $R$ and the fact that $B$ is $\pi$-adically complete.
\end{proof}

\begin{rmk} 
	\label{rmk:not-loc}
We remark that the base change functor $\FSmAlg_{R} \to \SmAlg_{\overline R}$ is not a localisation. Indeed, 
assume that $A \to B$ is a map in $\FSmAlg_{R}$ such that its base change $\overline A \to \overline B$ is an isomorphism. As $A$ and $B$ are $\pi$-torsion free and $\pi$-adically complete, it follows inductively that $A/\pi^{n}A \to B/\pi^{n}B$ is an isomorphism for all $n$ and hence so is $A \to B$. In other words, the above functor is conservative, i.e.~it does not invert any non-isomorphism. So if it was a localisation, it would be an equivalence, which is clearly not the case. We don't expect  the induced functor on presheaves of spaces $\PSh(\FSmAlg_{R}) \to \PSh(\SmAlg_{\overline R})$ to be a localisation either.
\end{rmk}

\section{Simplicially enriched formal schemes}

Let $R$ be as in Section~\ref{sec:liftsmo}. 
Let $\Sm_{\overline R}$ be the category of smooth $\overline R$-schemes and let $\SmAff_{\overline R}$ be the subcategory of smooth, affine schemes. 
Let $\FSmAff_{R} = \FSmAlg_{R}^{\op}$ be 
the category of affine, $\pi$-adic formal schemes over $R$ which are topologically of finite
type and formally smooth. The latter inherits a simplicial enrichment from $\FSmAlg_{R}^{\Delta}$, which we denote by $\FSmAff^{\Delta}_{R}$. 
We write $\Nerve(-)$ for the (simplicial) nerve functor (see~\cite[Def.~1.1.5.5]{HTT}).
As $\FSmAff^{\Delta}_{R}$  is fibrant by Proposition~\ref{prop:BCequiv}, its nerve  is an $\infty$-category \cite[Prop.~1.1.5.10]{HTT}. From Proposition~\ref{prop:BCequiv} we obtain a canonical equivalence 
\[
\Nerve(\FSmAff^{\Delta}_{R}) \xrightarrow{\sim} \SmAff_{\overline R}
\]
Using this equivalence, we equip the left-hand side with the Grothendieck topology corresponding to the Zariski or étale topology on $\SmAff_{\overline R}$. In the following, $\Sh(-)$ denotes the $\infty$-category of sheaves of spaces. 

\begin{prop}
	\label{prop:equiv-sheaves}
There are canonical equivalences
\[
\Sh(\Sm_{\overline R})  \xrightarrow{\sim} \Sh(\SmAff_{\overline R}) \xrightarrow{\sim} \Sh(\Nerve(\FSmAff^{\Delta}_{R})).
\]
\end{prop}
\begin{proof}
The functors are induced by pullback along $\Nerve(\FSmAff^{\Delta}_{R}) \xrightarrow{\sim} \SmAff_{\overline R} \to \Sm_{\overline R}$. Clearly, the second functor in the statement is an equivalence. As the Zariski/étale topology on $\Sm_{\overline R}$ has a basis consisting of affines, the first functor is an equivalence, too.
\end{proof}

\begin{rmk}
Proposition~\ref{prop:equiv-sheaves} is an analog of the constrution of the ``motivic Monsky–Washnitzer functor''
$\mathrm{DA}(k) \to \mathrm{RigDA}(K)$ due to Ayoub~\cite[Cor.~1.4.24, Rem.~1.4.25]{AyoubMotifsRig} and Vezzani~\cite{VezzaniMW}, where $K$ is a complete discretely valued field with
residue field $k$.
\end{rmk}

\section{PD-de Rham cohomology}
	\label{sec:pd-de-rham-cohomology}
	
Fix a prime number $p$.
Let $R$ be a ring, $\pi \in R$ an element, and assume that the ideal $(\pi)$ is endowed with a
PD-structure. Assume further that $p$ is nilpotent in $\overline R = R/(\pi)$ and that $R$ is $\pi$-adically complete.

\begin{defn}
We denote by $\PFSmAlg_R$ the following category. Its objects are pairs $(B,I)$ consisting of an $R$-algebra $B$ and a finitely generated ideal $I \subseteq B$, called ideal of definition, such that $B$ is $I$-adically complete and formally smooth over $R$, $I$ contains the image of $\pi$, and  $B/I$ is smooth over $\overline R$. The morphisms $\phi\colon (B,I)\to (B',I')$ are the $R$-algebra morphisms $\phi\colon B \to B'$
with $\phi(I)\subset I'$.
\end{defn}

Note that the category $\PFSmAlg_R$ has finite coproducts.

\begin{ex}\label{ex:smform}
Our basic example of an object in $\PFSmAlg_{R}$ is  
\begin{equation*}\label{eq:formsqgood}
(B, I ) = ( A\llbracket T_0 , \ldots , T_m\rrbracket, (\pi, T_{0}, \ldots, T_{m}) )
\end{equation*}
where $A$ is a $\pi$-adic complete $R$-algebra which is formally smooth and topologically of finite type over $R$.
\end{ex}

Let $ \mathrm{Ch}(R) $ be the $1$-category of complexes of $R$-modules and $
\mathrm{Ch}(R)^\wedge $ the subcategory of derived $\pi$-complete complexes of $R$-modules.

Consider the PD-de Rham complex as the functor $\dR\colon \PFSmAlg_R \to \mathrm{Ch}(R)^\wedge
$ defined as the
(derived) limit\footnote{As the transition maps are surjective, the naive limit in fact agrees with the derived limit.}
\[
\dR( B,I ) = \Omega^{*,\rm PD}_{(B,I)/(R,\pi)}  = \lim_n \Omega^{*}_{\mathrm{D}_{I}(B\otimes_R R_n) / R_n}
\]
where $R_n=R/(\pi^n)$ and where we take the PD-envelope of $I$ relative to the
PD-structure of $(\pi)$. The differential forms are the ones compatible with the
PD-structure \cite[Sec.~07HQ]{stacks-project}. If the ideal of definition $ I$ is clear from the context we
simply write $\dR(B)$ for $\dR(B,I)$. 

\begin{lemma}
	\label{lem:bc-PD-dR}
Assume that $\pi$ is a non-zero divisor in $R$.
All modules in the complex $\dR(B,I)$ are $\pi$-torsion free and there is an isomorphism of complexes $\dR(B,I) \otimes_{R} \overline R \cong \Omega^{*}_{\D_{I}(B \otimes_{R} \overline R)/\overline R}$.
\end{lemma}

\begin{proof}
Write $B_{n} = B \otimes_{R} R_{n}$, $\overline B = B_{1} = B/(\pi)$.
We first fix $n$ and $i$ and show that $\Omega^{i}_{\D_{I}(B_{n})/R_{n}}$ is a flat $R_{n}$-module which satisfies 
\[
\Omega^{i}_{\D_{I}(B_{n})/R_{n}} \otimes_{R_{n}} \overline R = \Omega^{i}_{\D_{I}(\overline B)/\overline R}.
\]
Lift $\overline C = B/I$ to a $\pi$-adically complete smooth $R$-algebra $C$, and let $C_{n} = C/(\pi^{n})$.
By the formal smoothness of $B$ we can lift the surjection $B \to B/I=\overline C$ to a
surjection $B \to C$. We get induced surjections $B_{n} \to C_{n}$ whose kernel we denote
by $J_{n}$. Note that $J_{1} = I\overline B$ is finitely generated by assumption. As $\pi$
is nilpotent in $B_{n}$ and $J_{n-1}\cong J_n /\pi^{n-1} J_n$ it follows that $J_{n}$ is finitely generated, too. 
As $B_{n}$ and $C_{n}$ are formally smooth over $R_{n}$, \cite[Cor.~$0_{\rm{IV}}$(19.5.4)]{EGA} implies that $J_{n}/J_{n}^{2}$ is a projective $C_{n}$-module and that there is an isomorphism $\Sym_{C_{n}}(J_{n}/J_{n}^{2})^{\wedge} \cong B_{n}$, where  ${}^{\wedge}$ indicates the completion for the augmentation ideal. By Zariski descent (noting that the divided power envelope commutes with localisation and is insensitive to the $I$-adic completion) we may  assume that $J_{n}/J_{n}^{2}$ is free, and hence $B_{n}$ is isomorphic to the power series ring $C_{n} \llbracket T_{1}, \dots, T_{r}\rrbracket$. It then follows from \cite[Rem.~3.20 1), 6)]{BerthelotOgus} that $\D_{I}(B_{n}) = \D_{J_{n}}(B_{n})$ and that one can drop the compatibility condition with the PD structure on $(\pi)$, and hence the PD envelope  is isomorphic to the PD polynomial algebra $C_{n}\langle T_{1}, \dots, T_{n} \rangle$. In particular, it is flat over $C_{n}$ and hence over $R_{n}$, and $\D_{I}(B_{n}) \otimes_{R_{n}} \overline R = \D_{I}(\overline B)/\overline R$. As $B_{n}$ is formally smooth over $R_{n}$ (and $p$ is nilpotent in $R_{n}$), $\Omega^{i}_{B_{n}/R_{n}}$ is a projective $B_{n}$-module. Thus 
\[
\Omega^{i}_{\D_{I}(B_{n})/R_{n}} \cong \Omega^{i}_{B_{n}/R_{n}} \otimes_{B_{n}} \D_{I}(B_{n})
\]
is flat over $R_{n}$ and obviously satisfies the base change formula asserted above.

Write $M_{n} = \Omega^{i}_{\D_{I}(B_{n})/R_{n}}$. As this is a flat $R_{n}$-module and
$R_{n}[\pi] = \pi^{n-1}R_{n}$ (as $\pi$ is a non-zero divisor in $R$), it follows that the
$\pi$-torsion submodule $M_{n}[\pi]$ is $\pi^{n-1}M_{n}$. Hence the pro-module
$\{M_{n}[\pi]\}$ is zero, and hence multiplication by $\pi$ is injective on $\lim_{n}
M_{n}$ with  cokernel $\lim_{n}M_{n}/\pi M_n = M_{1}$.
This is precisely the claim of the lemma.
\end{proof}

Berthelot proved a base change formula, the Poincaré lemma, and étale descent for PD-de Rham cohomology. We recall these in the form we need in the following proposition.

\begin{prop}\label{prop:basicpdder}
Assume that $\pi$ is a non-zero divisor in $R$.
  \begin{enumerate}
  \item The canonical map $\dR(B,I) \otimes_R^{\mathrm L} \overline R \xrightarrow{\sim} \dR(B/I)$ is an
  equivalence for $(B,I)$ in $\PFSmAlg_R$. Here the de Rham cohomology on the right is
  over  $\overline R$ with the element $\overline \pi = 0$. 

  \item
    For $(B,I)\to (B', I')$ a morphism in  $\PFSmAlg_R$
  such that $B/I\xrightarrow{} B'/I'$ is an isomorphism the induced map $\dR(B,I) \xrightarrow{\sim} \dR(B',I') $ is an equivalence.

  \item 
   Let  $(B,I) \to (C,K)$  be a map in $\PFSmAlg_{R}$ such that $B \to C$ is faithfully flat and étale, and $K$ is 
generated by the image of $I$. For the associated Cech nerve $B \to C^{\bullet}$ (equipped with the obvious simplicial ideal of definition $K^{\bullet}$)  we get an
  equivalence $\dR(B, I) \xrightarrow{\sim} \lim_\bullet \dR(C^{\bullet}, K^{\bullet})$.
  \end{enumerate}
\end{prop}

\begin{proof}
By Lemma~\ref{lem:bc-PD-dR} we have $\dR(B,I)\otimes_{R}^{\mathrm L} \overline R \simeq \Omega^{*}_{\D_{I}(\overline B)/\overline R}$.
We saw in the proof of the lemma that $\D_{I}(\overline B)$ is locally a PD polynomial algebra over $B/I$.
 As in
\cite[Lemma~07LD]{stacks-project}, the elementary Poincaré lemma
\cite[Lemma~V.2.1.2]{BerthelotCris} then implies that the canonical map $\Omega^{*}_{\D_{I}(\overline B)/\overline R} \to \Omega^{*}_{(B/I)/\overline R}$ is a quasi-isomorphism.

(2) As both complexes are derived $\pi$-complete, this follows immediately from (1).

(3) This follows from (1) and the fact that de Rham cohomology has étale descent.
\end{proof}

\section{A new approach to crystalline cohomology}\label{sec:newapcris}

In this section we construct a simplicially enriched version of the PD-de Rham complex functor and use it together with Proposition~\ref{prop:equiv-sheaves} to construct crystalline cohomology (Definition~\ref{def:crys}).

We start with some purely categorical preliminaries. Let $\cC$ be a category. We denote by $s\cC$ be the category of simplicial objects in $\cC$. Then $s\cC$ is canonically simplicially enriched; see \cite[discussion before Prop.~2.1.2]{Quillen67}. For $X_{\bullet}, Y_{\bullet} \in s\cC$, we denote the mapping simplicial set in $s\cC$ by $\Map_{s\cC}(X_{\bullet}, Y_{\bullet})$. Concretely, an $n$-simplex in $\Map_{s\cC}(X_{\bullet}, Y_{\bullet})$ is given by a collection of maps $f(\sigma) \colon X_{q} \to Y_{q}$ for every $q \geq 0$ and $\sigma\in \Delta[n]_{q}$, satisfying an obvious compatibility condition.
If $\cD$ is a second category, any functor $F\colon \cC \to \cD$ induces a simplicially enriched (simplicial, for short) functor $s\cC \to s\cD$.

Now assume that $\cC$ has finite coproducts. Fix a simplicial object $\Delta_{\bullet} \in s\cC$. We define a simplicial enrichment of $\cC$, denoted $\cC^{\Delta}$, via the formula
\[
\Map_{\cC^{\Delta}}(X,Y)_{n} := \Hom_{\cC}(X, Y \sqcup \Delta_{n}) = \Hom_{\cC_{\Delta_{n}/}}(X\sqcup \Delta_{n}, Y \sqcup \Delta_{n}) \qquad ([n]\in \Delta^{\op}).
\]
We denote by $s\cC_{\Delta_{\bullet}/}$ the simplicial slice category consisting of objects of $s\cC$ together with a map from $\Delta_{\bullet}$.

\begin{lemma}
	\label{lem:simplicial-enrichments}
The functor $\cC \to s\cC_{\Delta_{\bullet}/}$, $X \mapsto X \sqcup \Delta_{\bullet}$, extends to a simplicial functor $\cC^{\Delta} \to s\cC_{\Delta_{\bullet}/}$ and this simplicial functor is fully faithful, i.e.~it induces isomorphisms on mapping simplicial sets.
\end{lemma}

\begin{proof}
Fix $[n]\in \Delta^{\op}$. We have canonical isomorphisms 
\begin{align*}
\Map_{s\cC_{\Delta_{\bullet}/}}(X \sqcup \Delta_{\bullet}, Y\sqcup \Delta_{\bullet})_{n} &\cong 
	\Map_{s\cC}(X, Y\sqcup \Delta_{\bullet})_{n} \\
	&\cong \Hom_{\cC}(X, Y\sqcup \Delta_{n}) \\
	&= \Map_{\cC^{\Delta}}(X,Y)_{n}.
\end{align*}
The first isomorphisms follows from the definition of the slice category and the universal property of the coproduct. For the second, note that, as $X$ is a constant simplicial object, an $n$-simplex $f$ in $\Map_{s\cC}(X, Y\sqcup \Delta_{\bullet})$ is uniquely determined by $f(\id_{[n]}) \colon X \to Y\sqcup \Delta_{n}$. These isomorphisms are functorial in $[n] \in \Delta^{\op}$ and hence give the desired simplicial enrichment.
\end{proof}

We now return to the setting of the previous sections.   Let $R$ be a $\pi$-adically
complete ring where $\pi\in R$ is non zero-divisor, the ideal $(\pi)$ is equipped with a PD-structure, and $p$ is nilpotent in $R/(\pi)$.

Let $\D(R)$ denote the derived $\infty$-category of $R$, and $\D(R)^{\wedge}$ its subcategory formed by the derived $\pi$-complete complexes.

For $A \in \FSmAlg_{R}$, abusing notation slightly, we write $A\Delta^{\pi}_{\bullet}$ for the simplicial object 
\[
(A \widehat\otimes_{R} R\Delta^{\pi}_{\bullet}, (\pi, T_{0}, \dots, T_{\bullet})) \in s\PFSmAlg.
\]
 We will use as the abstract simplicial object $\Delta_\bullet$ above the simplicial
 object $R\Delta^{\pi}_{\bullet}$.

\begin{lemma}\label{lem:drdelta}
There is a functor 
\[
\dR^{\Delta}\colon \Nerve(\FSmAlg_{R}^{\Delta}) \to  \D(R)^{\wedge}
\]
which on objects is given by 
\[
\dR^{\Delta}(A) = \colim_{\Delta^{\op}} \dR(A\Delta^{\pi}_{\bullet}).
\]
\end{lemma}

\begin{proof}
The PD-de Rham complex as studied in the previous section defines a functor $\dR\colon \PFSmAlg_{R} \to \Ch(R)$ where $\Ch(R)$ denotes the category of chain complexes of $R$-modules. We now consider the composite of simplicial functors
\[
\FSmAlg_{R}^{\Delta} \to \PFSmAlg_{R}^{\Delta} \to s\PFSmAlg_{R,\Delta_{\bullet}^{\pi}/} \to s\PFSmAlg_{R} \to s\Ch(R).
\] 
Here the first functor is given by $A \mapsto (A, (\pi))$, the second is the one from Lemma~\ref{lem:simplicial-enrichments}, the third one is the canonical one, and the last one is induced by $\dR$. Applying the simplicial nerve, we obtain a functor of $\infty$-categories $\Nerve(\FSmAlg_{R}^{\Delta}) \to \Nerve(s\Ch(R))$. According to \cite[Prop.~1.3.4.7]{halg}, the latter category is the Dwyer-Kan-localization of the 1-category $s\Ch(R)$ at the class of simplicial homotopy equivalences. As the total complex functor $\mathrm{Tot}\colon s\Ch(R) \to \Ch(R)$ sends simplicial homotopy equivalences to quasi-isomorphisms, the composite $s\Ch(R) \to \Ch(R) \to \D(R)$ factors uniquely through a functor $t\colon \Nerve(s\Ch(R)) \to D(R)$. We can thus construct the desired functor $\dR^{\Delta}$ as the composite
\[
\Nerve(\FSmAlg_{R}^{\Delta}) \to \Nerve(s\Ch(R)) \xrightarrow{t} \D(R) \to \D(R)^{\wedge}
\]
where the last functor is the $\pi$-adic completion. 
As the total complex is a model for the (homotopy) colimit over $\Delta^{\op}$ (see e.g.~\cite[Cor.~3.13]{Arakawa:2023aa}), and as
the $\pi$-adic completion preserves colimits, we get the claimed description of the functor on objects.
\end{proof}

It follows from Proposition~\ref{prop:basicpdder}(3) that the functor $\dR^{\Delta}$ satisfies descent with respect to the étale topology. By the universal property of the $\infty$-category of sheaves, the Yoneda embedding induces an equivalence
\[
\Fun^{\mathrm{L}}(\Sh(\Nerve(\FSmAff^{\Delta}_{R}))^{\op}, \D(R)^{\wedge}) \xrightarrow{\simeq} \Fun^{\mathrm{desc}}(\Nerve(\FSmAlg^{\Delta}_{R}), \D(R)^{\wedge})
\]
where on the left-hand side we consider left adjoint functors and on the right-hand side functors satisfying étale descent. As the left-hand $\infty$-category identifies with $\Fun^{\mathrm{L}}(\Sh(\Sm_{\overline R})^{\op}, \D(R)^{\wedge})$ by Proposition~\ref{prop:equiv-sheaves}, the functor $\dR^{\Delta}$ uniquely extends to a functor $\dR^{\Delta} \colon \Sh(\Sm_{\overline R})^{\op} \to \D(R)^{\wedge}$.

\begin{defn}
	\label{def:crys}
We define crystalline cohomology as the functor
\[
\Cris\colon \Sm_{\overline R}^{\op} \to \D(R)^{\wedge}
\]
which is the composite of $\dR^{\Delta}$ with the Yoneda embedding $\Sm_{\overline R} \to \Sh(\Sm_{\overline R})$.
\end{defn}

\section{Comparison}

Let the notation be as above, i.e.\  $R$ be a $\pi$-adically complete ring, 
where $\pi\in R$ is a non-zero divisor, 
$(\pi)$ is endowed with a PD-structure,  and $p$ is nilpotent in $R/(\pi)$.

In this section we compare our functor $\mathrm{Cris}$ from Section~\ref{sec:newapcris} to
the de Rham functor $\dR\colon \fSm_R^\op \to \D(R)^\wedge$ on the category $\fSm_R$ of 
 $\pi$-adic formal schemes over $R$ which are topologically of finite type and formally
 smooth over $R$ and to Berthelot's crystalline
cohomology \cite{BerthelotCris}
\[
  \mathrm{BCris}\colon \Sm^\op_{\overline R}\to \D(R)^\wedge, \quad \overline
  X\mapsto  \lim_n R\Gamma_{\rm cris}(\overline X/R_n,\mathcal O^{\rm cris}).
  \]

  \begin{prop}
    There is a canonical isomorphism of functors
    \[
 \dR  \xrightarrow\sim \Cris\circ (- \otimes_R \overline R)\colon  \fSm_R^\op \to \D(R)^\wedge.
    \]
  \end{prop}

  \begin{proof}
    By Zariski descent it is sufficient to construct this isomorphism after restriction to
    $\FSmAff_R$. For $A$ in $\FSmAlg_R$ the functor on the left gives $\dR(A)$ and the
    functor on the right gives $\colim_{\Delta^\op} \dR(A\Delta^\pi_\bullet)$ by
    Lemma~\ref{lem:drdelta}. The canonical morphism of simplicial rings
    $A\to A\Delta^\pi_\bullet$ induces the requested natural transformation $ \dR  \to \Cris\circ (- \otimes_R \overline R)$. 
    Proposition~\ref{prop:basicpdder} implies that it is an equivalence.
  \end{proof}

\begin{prop}
  For $\overline X \in \Sm_{\overline R}$ there is a canonical isomorphism
  \[
\Cris(\overline X) \cong \mathrm{BCris}(\overline X).
\]
\end{prop}

\begin{proof}
By Zariski descent it is sufficient to construct this isomorphism after restriction to
$\SmAff_{\overline R}$.
Berthelot constructs a natural comparison isomorphism
\[
\mathrm{BCris}(B/I) \cong \dR(B,I) \text{ in } \D(R)^\wedge
\]
of his
crystalline cohomology to  PD-de Rham cohomology for $(B,I)$ in $\PFSmAlg_R$, see \cite[Thm.~V.2.3.2]{BerthelotCris},
\cite{BhattdeJong}, \cite[Prop.~07LG]{stacks-project}.
For any simplical object $(B_\bullet,I_\bullet)$ in  $\PFSmAlg_R$ we get an induced
isomorphism
\[
\colim_{\Delta^{\rm op}} \dR(B_\bullet, I_\bullet) \cong  \colim_{\Delta^{\rm op}} \mathrm{BCris}(B_\bullet/I_\bullet) .
\]
Taking $B_\bullet =
A\Delta_\bullet^\pi$ for $A$ in $\FSmAlg_R$ one gets the required equivalence.
\end{proof}

\bibliographystyle{amsalpha}
\bibliography{cryscoh}

\end{document}